\documentclass[11pt]{article}
\usepackage{amsmath,amssymb,amsthm,graphicx,subfigure,float,url}

\usepackage{tikz}
\usepackage{pdfsync}
\graphicspath{{fig/}}

\usepackage[colorlinks=true]{hyperref} 
\usepackage{pdfsync}
\usepackage{mathtools}
\DeclarePairedDelimiter\floor{\lfloor}{\rfloor}

\usepackage{enumitem}
\usepackage{dsfont} 

\usepackage[utf8]{inputenc}
\usepackage[T1]{fontenc}

\def\R{\mathbb{R}}
\def\N{\mathbb{N}}

\newcommand\norm[1]{\left\lVert#1\right\rVert}

\renewcommand{\geq}{\geqslant}
\renewcommand{\leq}{\leqslant}

\renewcommand{\geq}{\geqslant}
\renewcommand{\leq}{\leqslant}

\newcommand {\e}  {\varepsilon}

\newcommand {\Chi} {{\bf \raise 2pt \hbox{$\chi$}} }
\newcommand{\E}[1]{\mathbb{E}\left[{#1}\right]}
\newcommand{\supp}[1]{\operatorname{supp}({#1})}
\newcommand{\beq}{\begin{equation}}
\newcommand{\eeq}{\end{equation}}
\newcommand{\bea} {\begin{array}{rl}}
\newcommand{\eea} {\end{array}}
\newcommand{\bepa}{\left\{ \begin{array}{l}}
\newcommand{\eepa} {\end{array}\right.}

\newcommand{\bmu}{\begin{multline}}
\newcommand{\emu}{\end{multline}}

\newtheorem{theorem}{Theorem}  
\newtheorem{proposition}{Proposition}

\newtheorem{definition}{Definition}
\newtheorem{lemma}{Lemma}

\theoremstyle{definition}\newtheorem{remark}{Remark}

\title{Regularisation for the approximation of functions by mollified discretisation methods}

\date{}
\author{Marc Hoffmann\footnote{Universit\'e Paris Dauphine-PSL and Institut Universitaire de France, hoffmann@ceremade.dauphine.fr} \and Camille Pouchol\footnote{MAP5, UMR UMR 8145, Universit\'e Paris Cit\'e, camille.pouchol@u-paris.fr}}

\begin{document}

\newcounter{assum}

\maketitle

\begin{abstract}

Some prominent discretisation methods such as finite elements provide a way to approximate a function of $d$ variables from $n$ values it takes on the nodes~$x_i$ of the corresponding mesh.  The accuracy is $n^{-s_a/d}$ in $L^2$-norm, where $s_a$ is the order of the underlying method. When the data are measured or computed with systematical experimental noise, some statistical regularisation might be desirable, with a smoothing method of order $s_r$ (like the number of vanishing moments of a kernel). This idea is behind the use of some regularised discretisation methods, whose approximation properties are the subject of this paper. We decipher the interplay of $s_a$ and $s_r$ for reconstructing a smooth function on regular bounded domains from $n$ measurements with noise of order $\sigma$. We establish that for certain regimes with small noise $\sigma$ depending on $n$, when $s_a > s_r$, statistical smoothing is not necessarily the best option and {\it not regularising} is more beneficial than {\it statistical regularising}. We precisely quantify this phenomenon and show that the gain can achieve a multiplicative order $n^{(s_a-s_r)/(2s_r+d)}$. We illustrate our estimates by numerical experiments conducted in dimension $d=1$ with $\mathbb P_1$ and $\mathbb P_2$ finite elements. 
\end{abstract}

\vspace{5mm}


\section{Introduction}

\subsection{Motivation}
Let $\Omega$ be a smooth bounded connected open subset of $\R^d$ for some $d \geq 1$. We are interested in reconstructing a smooth function 
$$f : \overline{\Omega} \to \R$$ 
from its values on a fixed design given by $n$ points $x_i \subset \overline{\Omega}$. These values are moreover corrupted by noise. The points $x_i$ should be thought of as forming a mesh of the set~$\overline{\Omega}$. \\%

We focus on reconstruction methods that rely on \textit{regularised basis functions} or \textit{mollified basis functions}, as introduced in~\cite{Bramble1977, Thomee1977}.
Specifically, we are concerned with the case where functions are naturally (according to some given discretisation procedure) represented as linear combination of basis functions $\phi_i$: in other words, the function $f$ is approximated by 
\begin{equation}
\label{internal}
\sum_{i=1}^n f(x_i)\, \phi_i.
\end{equation}
Typical examples include discretisation of PDEs, where the $\phi_i$ are {\it e.g.} basis functions associated to $\mathbb{P}_k$ finite elements~\cite{ThomeeBookFEM2007, QuarteroniBook2009}. Informally, given a partition of $\overline{\Omega}$, the $\phi_i$ form a basis of the space of continuous functions on $\overline{\Omega}$ whose restriction to each piece of the partition is a polynomial of degree $k$. 

We will use the shorthand notation $u \lesssim v$ (or $v \gtrsim u$) whenever there exists a constant $C>0$ independent of $n$, $\sigma$ and $\beta$ (see below for a precise definition of the bandwidth parameter $\beta$ and the noise level $\sigma$) such that $u \leq C v$ for all $n>0$, $\sigma>0$ and $\beta>0$. We will write $A \sim B$ whenever $A \lesssim B$ and $B \lesssim A$ hold simultaneously. We also find it convenient to introduce a  \textit{discretisation parameter}  $h>0$ satisfying 
\[h \sim n^{-1/d}.\]
Of course, one could simply set $h = n^{-1/d}$, but in applications such as finite elements, there is a natural parameter $h$ which matches $n^{-1/d}$ up to multiplicative constants only.
\\ 

In the setting of~\eqref{internal}, one typically has an estimate of the form 
\[\sup_{f \in \mathcal F}\Big\| f- \sum_{i=1}^n f(x_i)\, \phi_i\Big\| \sim h^{s_a} \sim n^{-s_a/d}, \]
where $\|\cdot\|$ stands for the $L^2(\Omega)$-norm, $s_a>0$ for the order of the approximation method, and $\mathcal F$ a class of sufficiently smooth functions. 
In practice, because of measurement, numerical or roundoff errors, the sum $\sum_{i=1}^n f(x_i)\, \phi_i$ is rather given by 
$$\sum_{i=1}^n (y_\sigma)_i \, \phi_i,$$
with
\[(y_\sigma)_i = f(x_i)+ \sigma \xi_i,\;\;i = 1,\ldots, n,\] 
where our noise model is given by the $\xi_i$, assumed to be independent random variables, centred with unit variance, so that the parameter~$\sigma \geq 0$ quantifies the noise level as the common standard deviation to each measurement error.\\ 

A common and standard approach in alleviating the corresponding error is to operate some linear regularisation on the data given in the form $\sum_{i=1}^n (y_\sigma)_i \, \phi_i$, like {\it e.g.} convolution or projection onto low dimensional vector spaces. By {\it regularisation}, we mean that we are given a family of linear operators  $(R_\beta)_{\beta \geq 0}$ indexed by a smoothing parameter $\beta \geq 0$ such that $R_0 = \mathrm{Id}$ and {\it regularisation order} $s_r$. These typically satisfy estimates of the form
\[\sup_{f \in \mathcal F}\left\| R_\beta f- f\right\| \sim \beta^{s_r},\]
where, as before, $\mathcal F$ is a class of sufficiently smooth functions. 
This leads to estimators of the form
\[R_\beta \Big(\sum_{i=1}^n (y_\sigma)_i \, \phi_i\Big) =  \sum_{i=1}^n (y_\sigma)_i \,  R_\beta \phi_i,\]
and these natural candidate estimators for approximating $f$ are therefore based on the finite-dimensional subspace generated by the $n$ mollified basis functions $R_\beta \phi_i$.\\

 Of course, there are many other, potentially better, estimators at reconstructing $f$ from the data $(y_\sigma)_i$ without necessarily relying on regularised basis functions. There is immense literature on the subject in the field of nonparametric statistics; see {\it e.g.} the textbooks~\cite{GyorfiBook2002, TsybakovBook2008}. That our estimators are linear in particular means that one cannot hope for better approximation properties than those imposed by the Kolmogorov-$n$-width of the class $\mathcal{F}$~\cite{DevoreBook1993, LorentzBook1996}. \\
 
 Our main reason for sticking to this rigid reconstruction framework is that mollifying basis functions is actually quite common practice: such an approach dates back to the works~\cite{Bramble1977, Thomee1977} for parabolic equations. Indeed, these can lead to improved convergence estimates, and more pragmatically, they tend to stabilise the output. Extensions of this framework to hyperbolic equations also exist~\cite{Mock1978, Cockburn2003}, and these methods are still of current interest for applications~\cite{Febrianto2021}. The so-called Reproducing Kernel Element Method introduced in the series of papers~\cite{RKEM1, RKEM2, RKEM3, RKEM4} also relies on similar ideas, see Chapter 6 of~\cite{MeshfreeBook2007}. 
 \\ 
 
However, up to the best of our knowledge, the analysis of such methods does not include statistical errors such as the $\sigma \xi_i$ that are quantified in order by the standard deviation parameter $\sigma$, to be compared with $n$ or $h$. A natural question is therefore to understand how the presence of noise ({\it i.e.} $\sigma >0$ in our model) impacts the previous analysis. In particular, can we optimally quantify the interplay between $\sigma$ and $n$ (or equivalently between $\sigma$ and the mesh size $h$)? In other words, how best to mollify basis functions in the presence of noise, if mollifying is needed at all? This is the topic of the paper.

\subsection{Main results}

Given the setting and methodology described above, our overarching  goal is to choose a regularisation parameter $\beta$ appropriately as a function of the other parameters ({\it i.e.} $s_a, s_r, \sigma, n, d$), so that the reconstruction error when regularising at the order $\beta$ converges to $0$ as fast as possible as the number of observed data $n$ grows to infinity. Here the reconstruction error is defined by
\begin{equation}\label{eq: error def}
e(\beta, \sigma, n):=  \sup_{f \in \mathcal{F}(s,R)} \mathbb{E}\bigg[\Big\| f-R_\beta \Big(\sum_{i=1}^n  (y_\sigma)_i \, \phi_i\Big)\Big\|^2\bigg]^{1/2} ,
\end{equation}
where $\mathbb E[\cdot]$ denotes mathematical expectation w.r.t. the error distributions $(\xi_i)_{1 \leq i \leq n}$ and $\mathcal{F}(s,R)$ is a smoothness class of order $s>0$ in $L^2$, with radius $R>0$ (a Sobolev ball, see the precise definition~\eqref{sobolev_balls}). It is common statistical knowledge, see {\it e.g.} \cite{TsybakovBook2008}, \cite{MR1056335}, that a good choice for $\beta$ as a function of other parameters is given by
\begin{equation}
\label{common_reg}
\beta^\star(\sigma,n) \sim \sigma^{2/(2s_r+d)} n^{-1/(2s_r+d)}.
\end{equation}
as soon as $s \geq s_r$. The purpose of this work is to discuss regularisation strategies as functions of all involved parameters and to compare them to the common one given by~\eqref{common_reg}, or even to the possible strategy of possibly not regularising at all ({\it i.e.} when $\beta = 0$ and then $R_0 = \mathrm{Id}$).
 
We focus on a sufficiently simple and tractable setting as follows:
\begin{itemize}
\item We consider functions with sufficiently many derivatives vanishing on the boundary of $\Omega$ thus avoiding inessential boundary issues,
\item We quantify smoothness with a number of derivatives in $L^2$, hence considering Sobolev balls in~$H_0^s(\Omega)$,
\item We quantify estimation and reconstruction in integrated $L^2$-error loss, 
\item We restrict regularisation to the case of convolution with a kernel possessing vanishing moment properties.
\end{itemize}

The modelling framework developed in the present work could also serve as a stepping stone to analyse similar issues in the context of ill-posed inverse problems, \textit{i.e.} when one has access to (noisy approximations of) $A f(x_i)$ with $A$ a given compact operator from some Hilbert space to $L^2(\Omega)$. When $A$ is associated to an underlying partial differential equation, discretisation is naturally involved, while regularisation becomes necessary not only to cope with measurement errors, but also with the ill-posed nature of the problem~\cite{EnglBook1996, KirschBook2011}.

\subsubsection*{A general estimate}
We gather our two main results by means of informal statements; the precise hypotheses are to be found in Section~\ref{framework}. Our first result gives precise estimate of the error as a function of all the parameters.
\begin{theorem}
\label{general_error}
The error defined in \eqref{eq: error def} satisfies
\[e(\beta, \sigma, n)  \lesssim \sigma \min(\beta^{-1} n^{-1/d},1)^{d/2} + n^{-s_a/d}  + \beta^{s_r}.\]
In particular
 \begin{align*} \inf_{\beta>0} e(\beta, \sigma, n)  & \lesssim
\begin{cases} \sigma +  n^{-s_a/d} & \text{ if } \sigma \lesssim n^{-s_r/d}   \\
\sigma^{2 s_r/(2 s_r +d)}n^{-s_r/(2 s_r +d)} + n^{-s_a/d}  & \text{otherwise}
\end{cases}.
\end{align*}
\end{theorem}
These two estimates are given in Proposition~\ref{upper_est_error_reg}.
The first estimate, valid in the regime $\sigma \lesssim n^{-s_r/d}$, is obtained in the limit $\beta \to 0$. This is consistent with what can be achieved by {\it not regularising}, see Proposition~\ref{estimate_noreg}.
The second estimate, valid in the regime  $\sigma \gtrsim n^{-s_r/d}$, is obtained by choosing $\beta$ according to~\eqref{common_reg}.



\subsubsection*{The effect of not regularising versus regularising via~\eqref{common_reg}}
In order to compare the effect of not regularising versus regularising via~\eqref{common_reg}, we need lower bounds. 
We explicitly compare $\sigma$ and $n$ by writing $\sigma = \sigma(n) \sim n^{-\lambda/d} \sim h^{\lambda}$. 
The parameter $\lambda \geq 0$ quantifies the noise level, with $\lambda = 0$ corresponding to the largest possible noise level, {\it i.e.} when $\sigma$ is of order $1$. In this setting, the two errors we are interested in are given by 
$$e_{\mathrm{reg}}(n) := e\left(\beta^\star(\sigma(n),n), \sigma(n), n\right)$$
and
$$e_{\mathrm{noreg}}(n)  := e\left(0, \sigma(n), n\right),$$
corresponding to {\it regularising} (via~\eqref{common_reg}), or {\it not regularising} at all, {\it i.e.} ignoring the possible effect of the noise, deemed sufficiently negligible. Theorem \ref{general_error} establishes the existence of two regimes, depending on the relative positions of $s_a$ and $s_r$. 
In the case where $s_a \leq s_r$, it is always at least as good to regularise by means of the rule~\eqref{common_reg}, see Proposition~\ref{easy_regime}. This is a rather intuitive result, since regularisation in this case is of higher order, and hence cannot jeopardise the approximation property associated to discretisation.

The interesting situation is when $s_a >s_r$, in which case we uncover regimes when the option {\it not to regularise} is actually better! 
The interplay between the different parameters is a bit intricate: it involves the following non-standard threshold 
\begin{equation}
\label{weird_threshold}
\lambda_M :=  s_a + \frac{d}{2}\left(\frac{s_a}{s_r} - 1\right).
\end{equation}
More precisely, we obtain the following regimes depending on $\lambda_M$, as follows.
\begin{figure}[h]
 \begin{centering}
\includegraphics[totalheight=0.27\textheight]{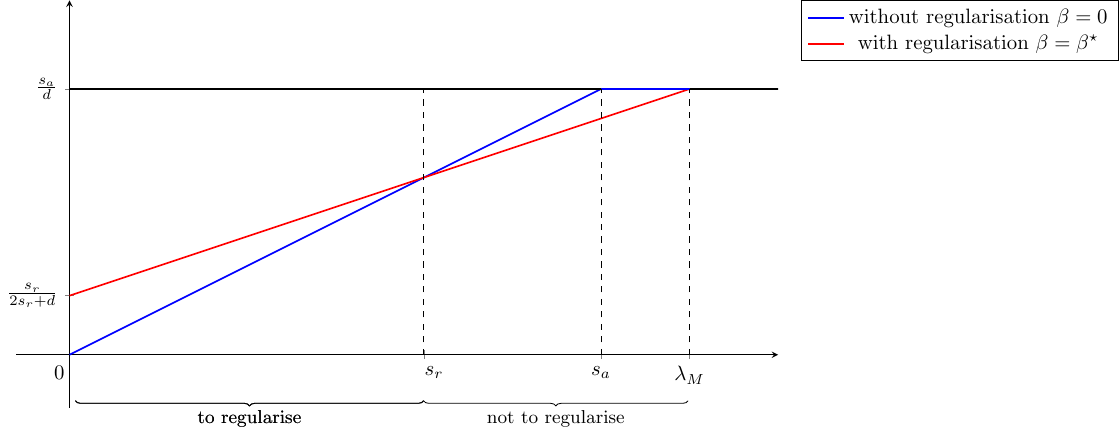}
\caption{{\it {\small For $0 \leq \lambda \leq \lambda_M$, plot of the order of convergence of $e_{\mathrm{noreg}}(n)$ and $e_{\mathrm{reg}}(n)$ towards $0$, as given by Theorem~\ref{interesting_regime}. In red, the function is $\lambda \mapsto \frac{1}{d}\min(\lambda, s_a)$, and in blue $\lambda \mapsto \frac{2\lambda+d}{2 s_r +d} \frac{s_r}{d}$. Parameters for this figure are chosen to be $d=2$, $s_a=3$, $s_r =2$, for which $\lambda_M = 3.5$.}}}
\label{fig_shakespeare}
\end{centering}
\end{figure}
\begin{theorem}
\label{interesting_regime}
Assume that $s_a > s_r$.
We have
\[ \begin{cases}e_{\mathrm{reg}}(n) \sim n^{-\frac{2\lambda+d}{2 s_r +d} \frac{s_r}{d}}  \text{ and } e_{\mathrm{noreg}}(n) \sim  n^{-\frac{\lambda}{d}}&  \text{ if } \; \lambda \leq s_a,    \\ \\
e_{\mathrm{reg}}(n)  \sim n^{-\frac{2\lambda+d}{2 s_r +d} \frac{s_r}{d}} \text{ and } e_{\mathrm{noreg}}(n)  \sim n^{-\frac{s_a}{d}} & \text{ if } \;  s_a < \lambda < \lambda_M, \\ \\
e_{\mathrm{reg}}(n)  \lesssim n^{-\frac{s_a}{d}}   \text{ and } e_{\mathrm{noreg}}(n)  \sim  n^{-\frac{s_a}{d}} & \text{ if }  \; \lambda \geq  \lambda_M.
\end{cases}.\]
\end{theorem}
Figure~\ref{fig_shakespeare} gives a schematic description of the situation when $s_a > s_r$. It depicts the order of convergence to $0$ of $e_{\mathrm{reg}}(n)$ and $e_{\mathrm{noreg}}(n)$, respectively, as a function of $\lambda$, in the regime $0 \leq \lambda \leq \lambda_M$. For such values of $\lambda$, Theorem~\ref{interesting_regime} yields 
\[e_{\mathrm{noreg}}(n) \sim n^{- \frac{1}{d} \min(\lambda, s_a)}, \qquad  e_{\mathrm{noreg}}(n) \sim n^{-\frac{2\lambda +d}{2 s_r +d} \frac{s_r}{d}}.\]
 The proofs are given in Proposition~\ref{sa>sr}.\\
 
  Several remarks are in order: {\bf 1)} Theorem \ref{interesting_regime} suggests the following alternative when having to choose between {\it not regularising} versus {\it regularising} through~\eqref{common_reg}: {\it regularise} through~\eqref{common_reg} whenever $\lambda < s_r$, but {\it do not regularise} whenever $s_r< \lambda < \lambda_M$.
{\bf 2)} It is easily seen that the highest gain in not regularising occurs for $\lambda =s_a$, value for which \[e_{\mathrm{reg}}(n) \sim  n^{-\frac{2 s_a+d}{2 s_r +d} \frac{s_r}{d}} 
\qquad e_{\mathrm{noreg}}(n)  \sim n^{-\frac{s_a}{d}}\]
One can hence gain up to the order $\textstyle \frac{s_a - s_r}{2 s_r +d}$.
{\bf 3)} We also have dependence of our estimates with respect to the dimension $d$. In the limit $d \rightarrow \infty$, the regime where {\it regularising} through~\eqref{common_reg} is optimal reduces to the single value $\lambda \in \{0\}$, whereas in the regime where {\it not regularising} is better, it becomes $\lambda \in (0,  \frac{1}{2}(\frac{s_a}{s_r} - 1))$. However, the gain  in {\it not regularising} through~\eqref{common_reg} vanishes in the limit $d \rightarrow \infty$, as the maximal gain $\textstyle \frac{s_a - s_r}{2 s_r +d}$ converges to $0$. {\bf 4)} Theorem~\ref{interesting_regime} is for instance relevant to the work~\cite{Febrianto2021}, where finite element methods of order up to $s_a = 4$ are regularised with nonnegative kernels, whose order cannot exceed (and actually equals) $s_r = 2$. 

\subsection*{Organisation of the paper}
In Section~\ref{framework}, we lay out the mathematical framework and provide all the hypotheses required for our main results Theorem \ref{general_error} and  \ref{interesting_regime} to hold.  Section~\ref{upper_bounds} gathers upper bounds for the errors either with $\beta = 0$ or with fixed $\beta>0$, which lead to Theorem~\ref{general_error}. We then compare the two main strategies, thanks to lower bounds at fixed $\beta$; these results are developed in Section~\ref{lower_bounds} and yield Theorem~\ref{interesting_regime}. Finally, Section~\ref{sec_num} is devoted to numerical experiments confirming our theoretical results, by means of examples in dimension $d=1$.

\section{Mathematical framework}
\label{framework}
We work in an arbitrary fixed dimension $d \in \N^*$, with $\Omega$ a smooth bounded connected open subset of $\R^d$.
We let $H^{s}(\Omega)$ denote the fractional Sobolev space of order $s\geq 0$, endowed with its natural norm $\|\cdot\|_s$ that corresponds (for $s \in \N$) to functions having $s$ distributional derivatives in $L^2(\Omega)$. The $L^2(\Omega)$-norm is written~$\|\cdot\|$ (rather than $\|\cdot\|_0$), with inner product $\langle \cdot, \cdot \rangle$. We let $H^s_0(\Omega)$ denote the closure of the space $C^\infty_c(\Omega)$ of infinitely differentiable compactly supported functions for the $\|\cdot\|_s$ norm. For basic definitions and results on fractional Sobolev spaces, we refer to the classical paper~\cite{GuideFractionalSobolev2012}.

\subsection{Statistical model and sampling}
We wish to reconstruct (equivalently estimate nonparametrically) a function $f \in H^s_0(\Omega)$ for $s>d/2$, from $n$ noisy measurements on a fixed design of $n$ points $x_i \in \overline{\Omega}$, with $i =1, \ldots, n$. 
Thanks to the Sobolev injection $H^s(\Omega) \hookrightarrow C^0(\overline{\Omega})$ valid for $s >d/2$~\cite{GuideFractionalSobolev2012}, the sampled values $f(x_i)$ are well-defined. We correspondingly define a sampling operator 
\begin{equation} \label{eq: def sampling op}
E_n : f \in  H^s_0(\Omega) \longmapsto (f(x_i))_{1 \leq i \leq n} \in \R^n.
\end{equation}
Our noisy measurements are given by the vector $y_\sigma  \in \R^n$ via the data
\[(y_\sigma)_i  = f(x_i)  +\sigma \xi_i = (E_n f)(x_i) + \sigma \xi_i, \quad i = 1, \ldots, n.\] Here, measurement noise is modelled by independent random variables $\sigma \xi_i$, $i = 1,\ldots, n$, where the $\xi_i$ are centred with unit variance.




\subsection{Discretisation}

Recall that the variable $h$ is related to $n$ by $h \sim n^{-1/d}$. We sometimes prefer to give our estimates in terms of $h$ rather than $n$, since the parameters $h$ and $\beta$ are homogeneous and therefore naturally compare. We assume that we are given a discretisation operator $P_n : \R^n \to L^2(\Omega)$ defined by means of \textit{basis functions} $\phi_i \in C^0(\overline{\Omega})$,  $i = 1, \ldots, n$.
via the reconstruction formula
\begin{equation} \label{eq: def discretisation op}
\forall z \in \R^n, \qquad P_n z = \sum_{i=1}^n z_i \phi_i.
\end{equation}
We will throughout assume that the basis functions are positive in a neighbourhood of size about $h$ around $x_i$, and vanish outside of a larger neighbourhood still of size about $h$. Our precise hypothesis reads as follows: there exist $m>0$, $C >c>0$ independent of $i$ and $n$ such that
\begin{equation}
\label{sup_base}
\phi_i(x_i  + h z)   \begin{cases} \geq m & \text{for } |z| \leq c \\
= 0  \quad &  \text{for } |z| \geq C.
\end{cases}
\end{equation}
All the symbols $\lesssim$ and $\sim$ below should also be understood to be uniform with respect to $i = 1, \ldots, n$.

In particular,  \eqref{sup_base} ensures the inclusion $B(x_i, c h) \subset  \supp{\phi_i} \subset B(x_i,  C h)$,
where $B(x_0,r) = \{x\in \overline{\Omega}, \, |x-x_0| \leq r\}$ denotes the closed Euclidean ball with center $x_0$ and radius $r\geq 0$.
We moreover assume
\begin{equation}
\label{pointwise}
\|\phi_i\|_{L^\infty(\Omega)} \lesssim 1,
\end{equation}
which in turn entails the estimates\footnote{
For the lower bounds, the first inequality of~\eqref{sup_base} entails $\int_{\Omega} |\phi_i(x)|^2\,dx  \geq m^2 |B(x_i, ch)| \sim h^d$,  $\int_{\Omega} |\phi_i(x)|\,dx  \geq m|B(x_i, ch)| \sim h^d$ which shows $\|\phi_i\| \gtrsim h^{d/2}$, and $\|\phi_i\|_{L^1(\Omega)} \gtrsim h^{d}$.  The uniform compact support given by~\eqref{sup_base} combined with~\eqref{pointwise} leads to $\|\phi_i\|^2 = \int_{\Omega} |\phi_i(x)|^2\,dx \lesssim |B(x_i, Ch)| \sim h^d$,  $\|\phi_i\|_{L^1(\Omega)} = \int_{\Omega} |\phi_i(x)|\,dx \lesssim |B(x_i, Ch)| \sim h^d$. }
\begin{equation}
\label{equal_basis}
\tag{$H_\phi^\sim$}
\|\phi_i\| \sim h^{d/2}\;\;\text{and}\;\;\|\phi_i\|_{L^1(\Omega)} \sim h^{d}.
\end{equation}
Note that the estimates~\eqref{equal_basis} are those essential for our results. We introduce the sufficient hypotheses~\eqref{sup_base} and~\eqref{pointwise} explicitly because they are more easily checked in practice. \\


We have a natural notion of accuracy of reconstruction that combine both the discretisation operator $P_n$ defined in \eqref{eq: def discretisation op} and the sampling operator $E_n$ defined in \eqref{eq: def sampling op}. 
\begin{definition} \label{def: orders}
We say that the discretisation-sampling pair $(P_n,E_n)$ has order (at least) $s_a >0$, if for every $s \geq s_a$,
\begin{equation}
\label{upper_dis}\norm{P_n E_n f - f} \lesssim \|f\|_{s_a} h^{s_a}.
\end{equation}
for every $f\in H_0^s(\Omega)$. 
We say that the discretisation-sampling pair $(P_n,E_n)$ has order exactly $s_a$ if
\begin{equation}
\label{equal_dis}
\tag{$H_a^\sim$}
\sup_{f \in \mathcal{F}(s,R)} \norm{P_n E_n f - f} \sim  h^{s_a},
\end{equation}
for every $R>0$, $s \geq s_a$.
\end{definition}
We use Sobolev balls as smoothness classes:
\begin{equation}
\label{sobolev_balls} \mathcal{F}(s,R) := \big\{f \in H^s_0(\Omega),  \; \norm{f}_{s} \leq R\big\},\;\;R>0.
\end{equation}
It is known that under fairly general hypotheses, $\mathbb{P}_k$ finite elements satisfy~\eqref{equal_dis} with $s_a = k+1$, see for instance \cite{ThomeeBookFEM2007}.

\subsection{Regularisation}
Pick a smooth and compactly supported kernel $K$ over $\R^d$, that satisfies in particular 
\[\int_{\R^d} K(x)\,d x = 1.
\]
We let $K_\beta := \beta^{-d} K(\beta^{-1} \cdot)$, and we note that 
\[\|K_\beta\|_{L^1(\R^d)} \lesssim 1, \quad \|K_\beta\|_{L^2(\R^d)} \lesssim \beta^{-d/2}.\]

 For a function $f \in L^2(\Omega)$, we define the convolution  
\[\forall x \in \R^d, \quad (K_\beta \ast f)(x) =  \int_\Omega K_\beta(x-y) f(y)\,dy.\]

We also assume that $K$ reproduces moments up to the degree $s_r -1\in \N^*$, but does not reproduce at least one moment of degree $s_r$, \textit{i.e.}, 
\begin{align}
\label{eq: def sr}
\begin{split}
\forall (r_1, \ldots, r_d) \in \N^d,  \; 1 \leq |r| \leq s_r-1, \qquad \int_{\R^d} x_{1}^{r_1} \ldots x_{d}^{r_d} \, K(x) \, dx = 0, \\
\exists (r_1, \ldots, r_d) \in \N^d,  \; |r| = s_r, \qquad \int_{\R^d} x_{1}^{r_1} \ldots x_{d}^{r_d} \, K(x) \, dx \neq 0.
\end{split}
\end{align}
 
Under the above assumptions and if $s \geq s_r$, we have
\begin{equation}
\label{upper_reg}
\sup_{f \in \mathcal{F}(s,R)}  \norm{K_\beta \ast  f - f}  \lesssim   \beta^{s_r}.
\end{equation}
In fact, the estimate above is sharp thanks to the assumption that $K$ does not reproduce at least one moment of degree $s_r$. In other words, for all $s \geq s_r$ and $R>0$ we have
\begin{equation}
\label{equal_reg}
\tag{$H_r^\sim$}
\sup_{f \in \mathcal{F}(s,R)}  \norm{K_\beta \ast  f - f}  \sim   \beta^{s_r}.
\end{equation}
Although these estimates are common, one is usually interested in the upper bound~\eqref{upper_reg}, with $\Omega = \R^d$ and integer parameter $s$.  For completeness, we thus provide a proof of~\eqref{equal_reg} in our setting, which we postpone to Appendix~\ref{app_fractional}.

\begin{remark}
Many common kernels (integrating to $1$) are nonnegative ($K \geq 0$) and symmetric $(K(x) = K(-x)$ for all $x \in \R^d$). The nonnegativity assumption prevents one from numerical instabilities.
However, these kernels are of order $s_r = 2$ and not more since some moments of order $2$ are not reproduced. 
\end{remark}



\subsubsection*{Recapitulating our assumptions}
From now on, we always assume that~\eqref{equal_basis},~\eqref{equal_dis} and~\eqref{equal_reg} hold.
Some of the results will in fact require weaker hypotheses, for instance in the form of upper bounds $\lesssim$ rather than equality $\sim$. The proof of each specific result will make it clear what is actually necessary for the claimed statement to hold.

\subsection{Reconstruction errors}
\subsubsection*{The case with no regularisation}
The first estimator is given by $P_n y_\sigma$. The corresponding error is 
\begin{equation} \label{eq: def noreg}
e_{\mathrm{noreg}}(\sigma, h) := \sup_{f \in \mathcal{F}(s,R)}\E{ \norm{ P_n y_\sigma- f}^2}^{1/2}. 
\end{equation}

\subsubsection*{The case with regularisation}
The second estimator consists in adding regularisation, the estimator being now given by $K_\beta \ast (P_n y_\sigma)$. The corresponding errors is 
$$e_{\mathrm{optreg}}(\sigma, h) := \inf_{\beta>0} e(\beta, \sigma, h),$$
with
$$e(\beta, \sigma,h) := \sup_{f \in \mathcal{F}(s,R)}\E{ \norm{K_\beta \ast P_n y_\sigma- f}^2}^{1/2}.$$

\section{Upper estimates}
\label{upper_bounds}
\subsection{The case with no regularisation ($\beta = 0$)}
We first analyse the error $e_{\mathrm{noreg}}(\sigma, h)$, associated to the estimator $P_n y_\sigma$ when no regularisation is involved.
\begin{proposition}
\label{estimate_noreg}
Assume that $s \geq s_a$. For every $R>0$, we have
\begin{align*}
e_{\mathrm{noreg}}(\sigma, h)\lesssim  \sigma +  h^{s_a}.
\end{align*}
\end{proposition}
\begin{proof}
Writing \[P_n y_{\sigma} - f =\left(P_n y_{\sigma} - P_n y\right) + \left(P_n y  - f\right) = P_n (y_{\sigma} - y) + \left(P_n E_n f  - f\right),\]
since the random variables $\xi_i$ are centred with unit variance, we obtain
\begin{align*}
\E{ \norm{ P_n y_\sigma- f}^2} & = \E{\norm{P_n (y_\sigma - y)}^2} +  \norm{ P_n E_n f- f}^2\\ 
& = \sigma^2 \mathbb E\big[\big\| \sum_{i=1}^n \xi_i \phi_i \big\|^2 \big]+  \|P_n E_n f- f\|^2 \\
& = \sigma^2 \sum_{i=1}^n \norm{ \phi_i}^2  +   \norm{ P_n E_n f- f}^2. 
\end{align*}
Owing to $\norm{ \phi_i}^2 \lesssim h^{d}$ which follows from~\eqref{equal_basis} we derive
\begin{align*}
\E{ \norm{ P_n y_\sigma- f}^2} \lesssim  \sigma^2 h^d n +   \norm{ P_n E_n f- f}^2 \sim \sigma^2 +   \norm{ P_n E_n f- f}^2.
\end{align*}
Using Assumption~\eqref{upper_dis}, taking square root and supremum over $f \in \mathcal{F}(s,R)$, we obtain the result.
\end{proof}

\subsection{The case with regularisation}
We now study $e_{\mathrm{optreg}}(\sigma, h)$, associated with the estimator $K_\beta \ast P_n y_\sigma$. 
Note first that Young's inequality yields
$$\norm{K_\beta \ast f} \leq \norm{K_\beta}_{L^1(\R^d)} \|f\|_{L^2(\Omega)}$$
and
$$\norm{K_\beta \ast f} \leq \norm{K_\beta}_{L^2(\R^d)} \|f\|_{L^1(\Omega)}.$$
\begin{proposition}
\label{upper_est_error_reg}
Assume that $s \geq \max(s_r, s_a)$. For every $R>0$, we have
\begin{equation} \label{eq: first bound}
e(\beta, \sigma, h)  \lesssim \sigma \min(\beta^{-1} h,1)^{d/2} + h^{s_a}  + \beta^{s_r}.
\end{equation}
In particular
 \begin{align*}e_{\mathrm{optreg}}(\sigma, h) & \lesssim
 \left\{
   \begin{array}{lll} 
 \sigma +  h^{s_a} & \text{ if } \sigma \lesssim h^{s_r},   \\ \\
\sigma^{2 s_r/(2 s_r +d)}h^{d s_r/(2 s_r +d)} + h^{s_a} & \text{ otherwise.}
\end{array}
\right.
 \end{align*}
\end{proposition}
The above alternative is obtained by letting $\beta \to 0$ and $\beta = \beta^\star(\sigma, h)$ respectively, with 
\[\beta^\star(\sigma, h)  \sim \sigma^{2/(2 s_r +d)}h^{d/(2 s_r +d)}.\]
\begin{proof}
In the same way as in the proof of Proposition \ref{estimate_noreg}, we have
\[\mathbb E\big[\big\|K_\beta \ast  P_n y_\sigma- f \big\|^2\big]  = \sigma^2 \sum_{i=1}^n \norm{ K_\beta \ast \phi_i}^2   + \norm{ K_\beta \ast P_n y- f}^2.\]
The second term may be estimated thanks to~\eqref{upper_dis} and~\eqref{upper_reg}. This yields 
\begin{align*} \norm{ K_\beta \ast P_n y- f} & \leq \norm{K_\beta \ast (P_n E_n f - f)} + \norm{K_\beta \ast f  - f} \\
&  \leq \norm{K_\beta}_{L^1(\R^d)} \norm{P_n E_n f -f}  + \norm{K_\beta \ast f  - f}  \\
& \lesssim  \norm{P_n E_n f -f}  + \norm{K_\beta \ast f  - f}   \\
&  \lesssim   \|f\|_{s_a} h^{s_a}  + \|f\|_{s_r}  \beta^{s_r}. 
\end{align*}
Two upper bounds may be derived for the first term, by means of two applications of Young's inequality, together with~\eqref{equal_basis}, namely
\[ \norm{K_\beta \ast \phi_i}^2 \leq \norm{K_\beta}_{L^1(\R^d)}^2  \norm{\phi_i}^2 \lesssim  \norm{\phi_i}^2 \lesssim h^{d}\]
and
\[ \norm{K_\beta \ast \phi_i}^2 \leq \norm{K_\beta}_{L^2(\R^d)} ^2  \norm{\phi_i}_{L^1(\Omega)}^2 \lesssim \beta^{-d} \norm{\phi_i}_{L^1(\Omega)}^2 \lesssim \beta^{-d}  h^{2d}.\]
The first choice leads to the following bound, valid for any $\beta>0$:
\begin{equation} \label{eq: second estimate}
e(\beta, \sigma, h)   \lesssim \sigma + h^{s_a}  + \beta^{s_r}.
\end{equation}
The second choice leads to 
\begin{equation} \label{eq: second estimate bis}
e(\beta, \sigma, h)   \lesssim \sigma \beta^{-d/2} h^{d/2} + h^{s_a}  + \beta^{s_r}.
\end{equation}
Combining the two estimates, we obtain \eqref{eq: first bound}.
Let us now minimise \eqref{eq: second estimate} and \eqref{eq: second estimate bis} with respect to $\beta$. For $\beta \lesssim h$, \eqref{eq: second estimate} is sharper while for $\beta \gtrsim h$ \eqref{eq: second estimate} prevails. For \eqref{eq: second estimate} we achieve a minimum or order $\sigma + h^{s_a}$ by letting $\beta \to 0$. Taking derivatives, \eqref{eq: second estimate bis}  is minimal for $\beta \sim \beta^\star(\sigma,h)$ with corresponding minimum of order $\sigma^{2 s_r/(2 s_r +d)}h^{d s_r/(2 s_r +d)} + h^{s_a}$. Finally, $\beta^\star(\sigma, h) \lesssim h$ if and only if $\sigma \lesssim h^{s_r}$, from which we infer 
\begin{align*}e_{\mathrm{optreg}}(\sigma, h) &= \inf_{\beta>0} \sup_{f \in \mathcal{F}(s,R)} \E{ \norm{K_\beta \ast P_n y_\sigma- f}_2^2}^{1/2}  \\
&  \lesssim \begin{cases} \sigma +  h^{s_a} & \text{ if } \sigma \lesssim h^{s_r}   \\ \\
\sigma^{2 s_r/(2 s_r +d)}h^{d s_r/(2 s_r +d)} + h^{s_a} & \text{ else}
\end{cases} \\
&  = \min(\sigma,\sigma^{2 s_r/(2 s_r +d)}h^{d s_r/(2 s_r +d)})+ h^{s_a}.
\end{align*}

\end{proof}

\section{Regularisation versus no regularisation}
\label{lower_bounds}
We wish to compare the effect of {\it not regularising} ({\it i.e.} $\beta =0$) versus {\it regularising}, with the (optimal) choice 
\begin{equation}
\label{reg_choice}
\beta = \beta^\star(\sigma,h) \sim \sigma^{2/(2 s_r +d)}h^{d/(2 s_r +d)}.
\end{equation}
In order to do so, we establish lower bounds for 
$e_{\mathrm{noreg}}(\sigma, h)$ defined in \eqref{eq: def noreg} and 
$$e_{\mathrm{reg}}(\sigma, h):= e(\beta^\ast(\sigma, h), \sigma, h).$$

\subsection{Estimates from below}
\begin{lemma}
\label{equal_noreg}
Assume $s \geq s_a$. For every $R>0$, we have: 
\[e_{\mathrm{noreg}}(\sigma, h) \sim \sigma + h^{s_a}.\]
\end{lemma}
\begin{proof}
Recall the identity  \[e_{\mathrm{noreg}}(\sigma, h)^2  =  \sigma^2 \sum_{i=1}^n \norm{ \phi_i}^2  +  \sup_{f \in \mathcal F_{s,R}}  \norm{ P_n E_n f- f}^2,\]
from which the result follows thanks to~\eqref{equal_basis} and~\eqref{equal_dis}.
\end{proof}

Observe in particular the identity
\[e(\beta, \sigma, h)^2   = \sigma^2 \sum_{i=1}^n \norm{ K_\beta \ast \phi_i}^2   + \sup_{f \in \mathcal F(s,R)}  \norm{ K_\beta \ast P_n y- f}^2.\]
\begin{lemma}
\label{lower_easy}
Let $s \geq \max(s_r,s_a)$. For every $R>0$, we have
\[e(\beta, \sigma,h) \gtrsim  \beta^{s_r/2} (\beta^{s_r/2} - h^{s_a/2}).\]
\end{lemma}
\begin{proof}
For $f \in \mathcal{F}(s,R)$, we write 
\begin{align*} \|K_\beta \ast P_n& y- f\|^2 = \norm{ K_\beta \ast (P_n E_n f-f) + (K_\beta \ast f -f)}^2 \\
&= \norm{ K_\beta \ast (P_n E_n f-f)}^2 + 2 \langle K_\beta \ast (P_n E_n f-f), (K_\beta \ast f -f)\rangle +   \norm{K_\beta \ast f -f}^2 \\
& \geq 2 \langle K_\beta \ast (P_n E_n f-f), (K_\beta \ast f -f)\rangle +   \norm{K_\beta \ast f -f}^2 \\ 
& \geq - 2\norm{ K_\beta \ast (P_n E_n f-f)}  \norm{K_\beta \ast f -f} +  \norm{K_\beta \ast f -f}^2  \\
&  \gtrsim - 2\norm{ P_n E_n f-f} \norm{K_\beta \ast f -f}  +  \norm{K_\beta \ast f -f}^2  \\
& \gtrsim -2 \|f\|_{s_a} h^{s_a} \|f\|_{s_r} \beta^{s_r}+\norm{K_\beta \ast f -f}^2\\
& \geq - 2 R^2 h^{s_a} \beta^{s_r} + \norm{K_\beta \ast f -f}^2,
  \end{align*}
  where we used both~\eqref{upper_dis} and~\eqref{upper_reg}. Thanks to~\eqref{equal_reg}, this yields
\begin{align*} \sup_{f \in \mathcal F(s,R)} \norm{ K_\beta \ast P_n y- f}_2^2 & \gtrsim- h^{s_a}\beta^{s_r} + \sup_{f \in \mathcal F(s,R)}\norm{K_\beta \ast f -f}_2^2 \\
& \gtrsim-  h^{s_a}\beta^{s_r} + \beta^{2 s_r} = \beta^{s_r} (\beta^{s_r} - h^{s_a}),
\end{align*}
and finally
\[e(\beta, \sigma,h) \gtrsim  \beta^{s_r/2} (\beta^{s_r/2} - h^{s_a/2}).\]
\end{proof}

\begin{remark}
A more comprehensive understanding of lower bounds for errors at fixed $\beta>0$ would notably require lower estimates for the norms $\|K_\beta \ast \phi_i\|$. We were only able to establish such estimates under restrictive assumptions, namely when $K \geq 0$ and assuming $\beta = \beta(h) = o(h)$. Since this result is only partial and does not happen to be necessary for the comparison between the two analysed strategies (not regularising or regularising through~\eqref{common_reg}), we delay these estimates until Appendix~\ref{app_further}.
\end{remark}

In order to compare $e_{\mathrm{noreg}}(\sigma, h)$ and $e_{\mathrm{reg}}(\sigma, h):= e(\beta^\ast(\sigma, h), \sigma, h)$ as functions of the noise level $\sigma$, we let 
\[\sigma = \sigma(h) \sim h^\lambda, \quad \text{ with} \quad \lambda \geq 0.\] 
It follows that 
\[\beta^\star(\sigma, h) \sim  \sigma^{2/(2 s_r +d)}h^{d/(2 s_r +d)} \sim h^{\frac{2\lambda+d}{2 s_r +d}}\] now only depends on $h$. For conciseness, we write $\beta^\star(h) = h^{\frac{2\lambda+d}{2 s_r +d}}$.

Both errors now depend on $h$ solely; abusing notation slightly, we write $e_{\mathrm{noreg}}(h)$ and $e_{\mathrm{reg}}(h)$, respectively.  Under~\eqref{equal_dis} and according to Lemma~\ref{equal_noreg}, 
\begin{align} \label{noreg_est}
\begin{split}
e_{\mathrm{noreg}}(h) \sim  \sigma(h) + h^{s_a} & \sim 
 \begin{cases} h^{\lambda}  & \text{ if } \lambda \leq s_a \\  \\ h^{s_a}  & \text{ if } \lambda > s_a\end{cases} \\
 &= h^{\min(\lambda, s_a)}.
 \end{split}
 \end{align}
We will also need the following elementary useful facts:
$$\beta^\star(h)^{s_r} \lesssim \sigma(h) \iff h\lesssim  \beta^\star(h) \iff  \lambda \leq s_r,$$
and 
$$h^{s_a} \lesssim \beta^\star(h)^{s_r} \iff \lambda \leq \lambda_M$$ where
\[\lambda_M := s_a + \frac{d}{2}\left(\frac{s_a}{s_r} - 1\right).\]
We note that if  $s_a \leq s_r$, we may have that $\lambda_M$ is negative and $\lambda_M\leq s_a \leq s_r$, whereas if $s_a >s_r$, then $s_r < s_a< \lambda_M$ always.

\begin{proposition}
\label{est_reg}
Assume that $s \geq \max(s_r,s_a)$. For every $R>0$, we have
\begin{align*}  \begin{cases} e_{\mathrm{reg}}(h) \sim  h^{\frac{2\lambda+d}{2 s_r +d} s_r}  & \text{ if } \lambda \leq \lambda_M,\\ \\
e_{\mathrm{reg}}(h) \lesssim h^{s_a}  & \text{ if } \lambda > \lambda_M.
\end{cases}
\end{align*}
\end{proposition}

\begin{proof}
Back to the estimate of Proposition~\ref{upper_est_error_reg}, we have
\begin{align*} e_{\mathrm{reg}}(h)&  \lesssim  \sigma(h) \min(\beta^\star(h)^{-1} h,1)^{d/2} + h^{s_a}  + \beta^\star(h)^{s_r}  = \min(\sigma(h), \beta^\star(h)^{s_r}) + h^{s_a}  + \beta^\star(h)^{s_r} \\
& \lesssim h^{s_a}  + \beta^\star(h)^{s_r},
& 
\end{align*}
and we infer
\begin{align*} e_{\mathrm{reg}}(h) \lesssim \begin{cases}  h^{\frac{2\lambda+d}{2 s_r +d} s_r}  & \text{ if } \lambda \leq \lambda_M\\ \\
h^{s_a}  & \text{ if } \lambda > \lambda_M,
\end{cases}
\end{align*}
It remains to show 
$$e_{\mathrm{reg}}(h) \gtrsim h^{\frac{2\lambda+d}{2 s_r +d} s_r}$$ 
whenever $\lambda \leq \lambda_M$. This is a consequence of Lemma~\ref{lower_easy} which gives 
$$e_{\mathrm{reg}}(h) \gtrsim \beta^\star(h)^{s_r/2} (\beta^\star(h)^{s_r/2} - h^{s_a/2}) \gtrsim \beta^\star(h)^{s_r} = h^{\frac{2\lambda+d}{2 s_r +d} s_r}$$ 
since $h^{s_a} \lesssim \beta^\star(h)^{s_r}$ under the assumption $\lambda \leq \lambda_M$.
\end{proof}

\begin{remark}
We do not know whether the tighter estimate $e_{\mathrm{reg}}(h) \sim h^{s_a}$ is valid for $\lambda > \lambda_M$ under our set of hypotheses, or if additional realistic assumptions can be made to establish it. 
\end{remark}

We now highlight situations where it is {\it strictly} more advantageous \textit{not to regularise} via the rule~\eqref{reg_choice} and ignore the effect of the regularisation.

\subsection{The case when $s_a \leq s_r$}
Our first result is that such a scenario does not occur whenever $s_a \leq s_r$.
\begin{proposition}
\label{easy_regime}
If $s_a \leq s_r \leq s$, then for every $R>0$,
\[ \begin{cases} e_{\mathrm{reg}}(h) \sim h^{\frac{2\lambda+d}{2 s_r +d} s_r}  \text{ and } e_{\mathrm{noreg}}(h) \sim h^\lambda& \text{ if } \lambda \leq \lambda_M,    \\ \\
e_{\mathrm{reg}}(h) \lesssim h^{s_a}  \text{ and } e_{\mathrm{noreg}}(h) \sim h^\lambda & \text{ if } \lambda_M < \lambda < s_a, \\ \\
e_{\mathrm{reg}}(h) \lesssim h^{s_a}   \text{ and } e_{\mathrm{noreg}}(h) \sim h^{s_a} & \text{ if } \lambda \geq s_a.
\end{cases}\]
In particular, it is strictly better to regularise through~\eqref{reg_choice} whenever $\lambda < s_a$, in which case we have
$$e_{\mathrm{reg}}(h)= o(e_{\mathrm{noreg}}(h)).$$ 
It is better to regularise through~\eqref{reg_choice} whenever $\lambda \geq s_a$,  in which case we have
$$e_{\mathrm{reg}}(h)  \lesssim e_{\mathrm{noreg}}(h).$$
\end{proposition}
\begin{proof}
Recall that $\lambda_M\leq s_a \leq s_r$. All cases are obtained by combining Proposition~\ref{est_reg} with the estimate~\eqref{noreg_est}.
\end{proof}

\subsection{The case when $s_a > s_r$}

In that case, it is indeed possible to find situations where it becomes {\it strictly} more advantageous \textit{not to regularise} via the rule~\eqref{reg_choice} and ignore the effect of the regularisation, a perhaps surprising result.

\begin{proposition}
\label{sa>sr}
If $s \geq s_a > s_r$, then for every $R>0$,
\[ \begin{cases} e_{\mathrm{reg}}(h) \sim h^{\frac{2\lambda+d}{2 s_r +d} s_r}  \text{ and } e_{\mathrm{noreg}}(h) \sim h^\lambda& \text{ if } \lambda \leq s_a,   \\  \\
e_{\mathrm{reg}}(h)  \sim h^{\frac{2\lambda+d}{2 s_r +d}s_r} \text{ and } e_{\mathrm{noreg}}(h) \sim h^{s_a} & \text{ if } s_a < \lambda < \lambda_M, \\ \\
e_{\mathrm{reg}}(h) \lesssim h^{s_a}   \text{ and } e_{\mathrm{noreg}}(h) \sim h^{s_a} & \text{ if } \lambda \geq  \lambda_M.
\end{cases}\]
In particular, it is strictly better to regularise  through~\eqref{reg_choice}  whenever $\lambda < s_r$, in which case we have
$$e_{\mathrm{reg}}(h)= o(e_{\mathrm{noreg}}(h)).$$
It is strictly better not to regularise  through~\eqref{reg_choice}  whenever $s_r < \lambda  <\lambda_M$, in which case we have
$$e_{\mathrm{noreg}}(h)= o(e_{\mathrm{reg}}(h)).$$
Finally, it is better to regularise  through~\eqref{reg_choice} whenever $\lambda \geq \lambda_M$, in which case we have
 $$e_{\mathrm{reg}}(h)  \lesssim e_{\mathrm{noreg}}(h).$$
\end{proposition}
\begin{proof}
Recall that $s_r < s_a < \lambda_M$.  Again, all cases are obtained by combining Proposition~\ref{est_reg} with the estimate~\eqref{noreg_est}.
\end{proof}




As mentioned earlier, we can go further and estimate the level of noise with  highest gain {\it in not regularising} compared to {\it regularising} in the regime  $s_a > s_r$. We find it more transparent to express this gain in terms of sampling size $n$ rather than in terms of the mesh size $h$, since $n$ may be regarded as the actual cost of measuring $f$ over the design~$x_i$, $i =1, \ldots, n$. Recall that $\sigma(h) = \sigma(n) \sim n^{-\frac{\lambda}{d}}$ with $\lambda \geq 0$. 
The highest gain happens when $\lambda=s_a$ for which \[e_{\mathrm{reg}}(n) \sim h^{\frac{2 s_a+d}{2 s_r +d} s_r} \sim n^{-\frac{2 s_a+d}{2 s_r +d} \frac{s_r}{d}}, \quad e_{\mathrm{noreg}}(n) \sim h^{s_a} \sim n^{-\frac{s_a}{d}}.\] One can hence gain up to a polynomial (in $n$) factor of order $\textstyle \frac{s_a - s_r}{2 s_r +d}$ which vanishes for large~$d$. This is consistent with the condition $s > d/2$ which somehow enforces $f$ to be smoother as $d$ increases. 





 
 \section{Numerical simulations}  
 \label{sec_num}

 \subsection{Setting}
We work in dimension $d = 1$ with $\Omega = (0,1)$. We are mostly interested in situations where regularisation might be detrimental, \textit{i.e}, when $s_r < s_a$. Hence, we choose kernels of order $s_r = 1$ and $s_r = 2$ respectively, and approximation methods of order $s_a = 2$ and $s_a = 3$. These are defined below. 
\subsubsection*{Regularisation}
We consider two kernels $K$ and $H$, given by 
$$K := \mathds{1}_{[0,1]}, \qquad \text{and}\;H :=\textstyle \frac{1}{2} \mathds{1}_{[-1,1]}.$$
They satisfy $s_r = 1$ and  $s_r = 2$ respectively. The kernel $K$ is not standard: it is not centred, hence its low order of convergence. We make this rather artificial choice in order to better illustrate our results which are most visible when the gap $s_a - s_r$ gets larger, especially in small dimensions.

\subsubsection*{Approximation}
We consider $\mathbb{P}_1$ and $\mathbb{P}_2$ finite elements. We make sure to be consistent with our choice that~$n$ represents the number of basis functions. In doing so, the definitions below slightly differ from usual definitions which have $h$ rather than $n$ as the defining parameter. \\

\noindent \textit{$\mathbb{P}_1$ finite elements.} 
We let $n \geq 3$ be given. We define $h := \frac{1}{n-1}$, and for $i = 1, \ldots, {n}$, we denote $x_i := (i-1)h$.
Defining the shape function \[\forall x \in [0,1], \qquad \varphi(x) := (1-|x|) \mathds{1}_{[-1,1]}(x),\] the basis functions are then given as follows, for $i = 1, \ldots, n$:
\[\forall x \in [0,1], \qquad \phi_i(x) = \varphi\Big(\frac{x-x_i}{h}\Big). 
\]
These basis functions clearly satisfy~\eqref{sup_base} and~\eqref{pointwise}, so that~\eqref{equal_basis} holds. Furthermore, the approximation operator $P_n E_n$ associated to $\mathbb{P}_1$ finite elements satisfies $s_a = 2$.\\ 

\noindent \textit{$\mathbb{P}_2$ finite elements.} 
We let $n \geq 3$ be an odd interger. We define $h := \frac{2}{n-1}$, and for $i = 1, \ldots, {n},$, we set $\textstyle x_i := (i-1)\frac{h}{2}$.
Defining the shape functions \[\forall x \in [0,1], \qquad \varphi(x) := (1-|x|) (1-2 |x|) \mathds{1}_{[-1,1]}(x),\quad \psi(x) := (1-4x^2) \mathds{1}_{[-\frac{1}{2},\frac{1}{2}]}(x),\] the basis functions are then given as follows, for $i =1, \ldots, n$:
\[\forall x \in [0,1], \qquad \phi_i(x) =\begin{cases}  \varphi\big(\frac{x-x_i}{h}\big) & \text{ if $i$ is odd,}   \\ \\
 \psi\big(\frac{x-x_i}{h}\big) & \text{ if $i$ is even.}  
 \end{cases}
\]
These basis functions clearly satisfy~\eqref{sup_base} and~\eqref{pointwise}, so that~\eqref{equal_basis} holds.
The approximation operator $P_n E_n$ associated to $\mathbb{P}_2$ finite elements satisfies $s_a = 3$.

 \subsection{Methodology}
 In order to illustrate our theoretical results, we aim at computing, for a given function $f \in H^s_0(\Omega)$ with $s > d/2 =1/2$, the two errors for various noise levels $\sigma = \sigma(h) = h^\lambda$, which corresponds to $\sigma =\sigma(n) = n^{-\lambda/d}$. 
More precisely, we are interested in finding how quickly the two errors
\begin{equation*} 
 \E{ \norm{ P_n y_\sigma- f}^2}^{1/2} \quad \text{and} \quad 
 \label{error_reg}\E{ \norm{ K_{\beta^\star(h)} P_n y_\sigma- f}^2}^{1/2}\end{equation*}  vanish as $n$  grows, as a function of the noise level defined by the parameter~$\lambda$. Recall that the regularisation is made with a parameter $\beta$ chosen to be $\beta^\star(h)$ given by~\eqref{reg_choice}. Mathematically, for a given choice of $\lambda$, a given error is of  order $n^{-\gamma(\lambda)}$ and our goal is to estimate the function $\gamma$ over a given interval for $\lambda$. Hence, for a given choice of approximation method, we have a function $\lambda \mapsto \gamma_{\text{noreg}}(\lambda)$ defined by
 \begin{equation} 
 \label{error_noreg}
 \E{ \norm{ P_n y_\sigma- f}^2}^{1/2} \sim n^{-\gamma_{\text{noreg}}(\lambda)},
 \end{equation}
and for a given choice of kernel and approximation method, we have a function $\lambda \mapsto \gamma_{\text{reg}}(\lambda)$ defined by
 \begin{equation} 
 \label{error_reg}\E{ \norm{ K_{\beta^\star(h)} P_n y_\sigma- f}^2}^{1/2} \sim n^{-\gamma_{\text{reg}}(\lambda)} .\end{equation} 
From our theoretical results, recalling that $d = 1$, we expect
\[\forall \lambda \geq 0, \quad \gamma_{\text{noreg}}(\lambda) = \min(\lambda, s_a), \quad \text{ and }\quad  \forall \lambda \in [0,\lambda_M], \quad \gamma_{\text{reg}}(\lambda) = \frac{2\lambda+1}{2 s_r +1} s_r.\]
For $\lambda > \lambda_M$, we recall our upper bound for the error which translates into the lower bound $\gamma_{\text{reg}}(\lambda) \geq s_a$. When plotting functions $\gamma_{\text{noreg}}$ and $\gamma_{\text{reg}}$, we pay specific attention to the regime \[\lambda \in [0,\lambda_M] = \left[0, \, s_a + \frac{1}{2}\left(\frac{s_a}{s_r} - 1\right)\right],\]
but we will consider the larger interval $[0,5]$; the latter contains $[0, \lambda_M]$ in all cases.

When estimating the error without regularisation, we consider both  $\mathbb{P}_1$ and $\mathbb{P}_2$ finite elements.\\

 When estimating the error with regularisation, we consider all $4$ possible scenarios, corresponding to choosing the kernel to be either $K$ or $H$, and the approximation method to be either $\mathbb{P}_1$ or $\mathbb{P}_2$ finite elements. Note that only in the case where the kernel is $H$ and with $\mathbb{P}_1$ finite elements does one have $s_r = s_a$; in all other cases $s_r < s_a$.

\subsubsection*{Estimating orders of convergence} For a fixed choice of $\lambda \geq 0$, we must evaluate how quickly a given error tends to $0$ as a functions of $n$. In order to do so, we choose $n = 10, 10^2, 10^3$ (which corresponds to $h\sim 10^{-1}, 10^{-2}, 10^{-3}$ respectively), and compute the slope of both errors in $\log$-$\log$ scale.

\subsubsection*{Estimating expectations} 
For a fixed choice of $n$, this means we have to compute the above errors; we evaluate the expectations by means of $10^3$ draws for the random variables~$\xi_i$ (chosen to be normally distributed). 

\subsubsection*{Estimating norms} For a given draw, $L^2$-norms $\|\cdot\|$ are estimated by Simpson's rule with $10^5$ points, in order to ensure accurate estimates that do not compete with the expected orders of convergence. 

\subsubsection*{Estimating convolutions} Let us stress that other integrals are involved in the process of computing the error, when regularisation is involved. Those are inherent in evaluating the convolution $K_\beta \ast P_n y_\sigma$, which in turn boils down to evaluating all functions $K_\beta \ast \phi_i$, $i = 1, \ldots, n$. In order for these computations to not impact the orders of convergence, we analytically rather than numerically compute these functions. This is possible for our choices of kernels and finite element functions. 

In practice, however, these integrals would be computed with errors. All other things being equal, these errors can only further reduce the quality of regularising compared to not regularising.

 \subsection{Numerical results}
 For all numerical experiments, we choose 
 \begin{equation}
 \label{choice_function}
f : x \mapsto (1-x)^2 \sin^2(4x)
 \end{equation}
 which satisfies $f \in H_0^s(\Omega)$ with $s > 3$, so that we will always have $s \geq s_r$ as well as  $s \geq s_a$. 
 
 \subsubsection*{$\mathbb{P}_1$ and  $\mathbb{P}_2$ finite elements}
The orders of convergence obtained numerically match the theoretical ones, as shown by Figure~\ref{noreg}. One indeed expects the function $\gamma_{\text{noreg}}(\lambda) : \lambda \mapsto \min(\lambda, s_a)$ and this is exactly what is found.

\begin{figure}
     \centering
     \begin{subfigure}
         \centering
         \includegraphics[width=0.47\textwidth]{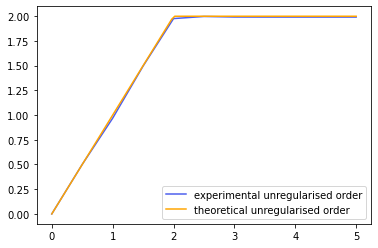}
     \end{subfigure}
     \begin{subfigure}
         \centering
         \includegraphics[width=0.47\textwidth]{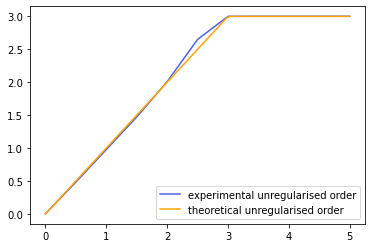}
     \end{subfigure}
        \caption{{\small{\it  \textbf{No regularisation}: plot of $\lambda \mapsto \gamma_{\text{noreg}}(\lambda)$ defined by~\eqref{error_noreg}. The left panel shows the case of $\mathbb{P}_1$ finite elements, the panel right that of $\mathbb{P}_2$ finite elements. In both cases, the theoretical curve $\lambda \mapsto \min(\lambda, s_a)$ is plotted in orange against the numerically obtained curve, in magenta.}}}
        \label{noreg}
\end{figure}

\subsubsection*{The kernel $K$ with $\mathbb{P}_1$ and $\mathbb{P}_2$ finite elements}
In this case, $s_r = 1$, with either $s_a  = 2$ or $s_a = 3$. In the case of $\mathbb{P}_1$ finite elements, one has $\lambda_M = 2.5$, while in the second $\lambda_M =4$. The orders of convergence obtained numerically are a good match to the theoretical ones, as shown by Figure~\ref{reg1}. The match is almost perfect in the $\mathbb{P}_1$ case. In the second case of $\mathbb{P}_2$ finite elements, discrepancies may be observed as $\lambda$ approaches $\lambda_M = 4$.
\begin{figure}
     \centering
     \begin{subfigure}
         \centering
         \includegraphics[width=0.47\textwidth]{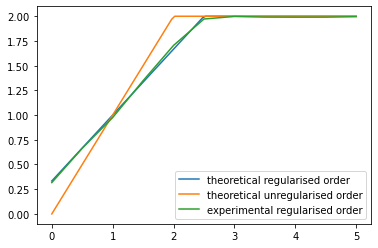}
     \end{subfigure}
     \begin{subfigure}
         \centering
         \includegraphics[width=0.47\textwidth]{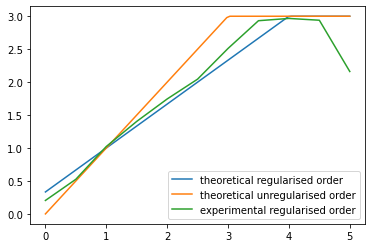}
     \end{subfigure}
        \caption{{\small {\it \textbf{Regularisation with kernel $K$}: plot of $\lambda \mapsto \gamma_{\text{noreg}}(\lambda)$ defined by~\eqref{error_noreg}, with regularisation through the kernel $K$. The left panel shows the case of $\mathbb{P}_1$ finite elements ($\lambda_M = 2.5$), the panel right that of $\mathbb{P}_2$ finite elements ($\lambda_M = 4$). In both cases, the theoretical curves without regularisation $\lambda \mapsto \min(\lambda, s_a)$ and with regularisation $\lambda \mapsto \frac{1}{3}(2\lambda +1)$ are plotted in orange and blue, respectively. The numerically obtained curve (with regularisation through $K$) is plotted in green.}}}        \label{reg1}
\end{figure}

\subsubsection*{The kernel $H$ with $\mathbb{P}_1$ and $\mathbb{P}_2$ finite elements}
In this case, $s_r = 2$, with either $s_a  = 2$ or $s_a = 3$. In the first case, one has $\lambda_M = s_r = s_a = 2$, while in the second $\lambda_M = 3.25$. The orders of convergence obtained numerically are a good match to the theoretical ones, as shown by Figure~\ref{reg2}. In both cases, there is little difference between the two theoretical curves, making it more difficult to clearly distinguish the numerically-built curve from the two theoretical ones.\\
\begin{figure}
\label{reg2}
     \centering
     \begin{subfigure}
         \centering
         \includegraphics[width=0.47\textwidth]{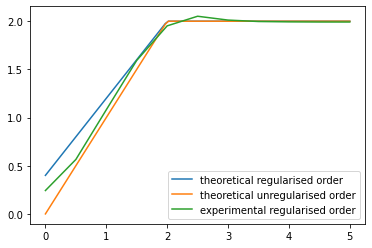}
     \end{subfigure}
     \begin{subfigure}
         \centering
         \includegraphics[width=0.47\textwidth]{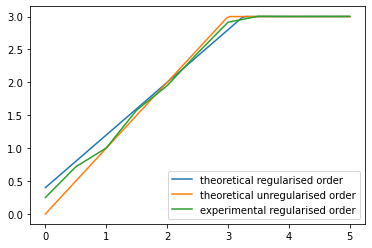}
     \end{subfigure}
        \caption{\textbf{Regularisation with kernel $H$}: plot of $\lambda \mapsto \gamma_{\text{reg}}(\lambda)$ defined by~\eqref{error_noreg}, with regularisation through the kernel $H$. The left panel shows the case of $\mathbb{P}_1$ finite elements ($\lambda_M = 2$), the panel right that of $\mathbb{P}_2$ finite elements ($\lambda_M = 3.25$). In both cases, the theoretical curves without regularisation $\lambda \mapsto \min(\lambda, s_a)$ and with regularisation $\lambda \mapsto \frac{2}{5} (2\lambda+1)$ are plotted in orange and blue, respectively. The numerically obtained curve (with regularisation through $H$) is plotted in green.}        \label{reg2}
\end{figure}

In this 1-dimensional setting, the actual improvement obtained by {\it not regularising} whenever this is superior to the regularisation by~\eqref{common_reg} is hardly visible at the level of reconstructions. This is why we do not provide examples of such reconstructions.

%
%
%
%
%

\bibliographystyle{alpha}
\bibliography{Biblio_Shakespeare.bib}
\appendix
\section{A proof of the estimate~\eqref{equal_reg}}
\label{app_fractional}
We here prove the estimate~\eqref{equal_reg}. We let $s \geq s_r$.

\subsubsection*{Lower bound}
We start with the easiest part, namely the lower bound, which comes from the assumption that $K$ does not reproduce one moment of order $s_r$, which we denote $P(x) = x_1^{r_1}\ldots x_d^{r_d}$ with $r_1 + \ldots + r_d = s_r$. Without loss of generality, we may assume that $0 \in \Omega$ and we consider the function $f = \chi P$ where $\chi \in C_c^\infty(\Omega)$ equals $1$ in a neighbourhood of $0$. Hence we have $f \in H^s_0(\Omega)$.

 For $\beta$ small enough and a sufficiently small neighbourhood of $0$ (which we denote by~$\Omega_0$) we may write $K_\beta\ast f(x) - f(x) = K_\beta \ast P(x) -P(x)$ for all $x \in \Omega_0$, where we use that $K$ has compact support.
 For $x \in \Omega_0$ and $\beta$ small enough so that $\beta^{-1} (x-\Omega) \subset \supp{K}$ for all $x \in \Omega_0$, we have
\begin{align*} 
K_\beta \ast P(x) -P(x) & = \int_{\beta^{-1}(x-\Omega)} K(u) P(x-\beta u) \,du - P(x) \\
& = \int_{\supp{K}} K(u) P(x-\beta u) \,du - P(x). 
\end{align*}
When expanding the product $P(x-\beta u) = (x_1-\beta u_1)^{r_1} \ldots (x_d- \beta u_d)^{r_d}$ and integrating against $K$, all terms but two vanish since $K$ reproduces moments up to order $s_r-1$, and we are left with 
\begin{align*} 
K_\beta \ast P(x) -P(x) & = \int_{\supp{K}} K(u) (P(x)  +(-1)^{s_d} \beta^{s_r} P(u)) \,du - P(x) \\
& = P(x) \left( \int_{\supp{K}} K(u)  \,du -1\right) +(-1)^{s_d} \beta^{s_r} \int_{\supp{K}} K(u) P(u) \,du \\
& = (-1)^{s_d}  \beta^{s_r} \int_{\supp{K}} K(u) P(u) \,du,
\end{align*}
where the last constant appearing is non-zero by assumption. As a result, we may write 
\[\|K_\beta \ast f - f\|_{L^2(\Omega)} \geq \|K_\beta \ast f - f\|_{L^2(\Omega_0)} \gtrsim \beta^{s_r}.\]
Uo to changing $f$ to $\tfrac{R}{\|f\|_{s_r}} f$, we thus have found some $f \in \mathcal{F}(s,R)$ such that $\|K_\beta \ast f - f\|\gtrsim \beta^{s_r}$, and it follows that 
\[\sup_{f \in \mathcal{F}(s,R)}  \norm{K_\beta \ast  f - f} \gtrsim   \beta^{s_r}.\]

\subsubsection*{Upper bound}

Now let $f \in H^s_0(\Omega)$.
We start with the case of $\Omega = \R^d$.
For $x \in \R^d$, we have
\[ K_\beta \ast f(x) - f(x)  = \int_{\R^d} \left( f(x- \beta y)-f(x)\right) \beta^{-d}K(\beta^{-1}y) \,dy = \int_{\R^d} \left(f(x- \beta y)-f(x)\right)K(y)\,dy.\]
Since $K$ has $\floor{s}$ vanishing moments, we may replace $f(x)$ by the Taylor polynomial of $f$ of order $\floor{s}$ at the point $x$, evaluated at $-\beta y$, which we denote $P_{\floor{s}}(x, -\beta y)$.
Hence we find\[f(x) - K_\beta \ast f(x)  = \int_{\R^d} \left(f(x- \beta y) - P_{\floor{s}}(x, -\beta y) \right)K(y)\,dy.\]
Next, we apply Minkowski's integral inequality (for the Lebesgue measure  $dx$ and the measure $|K(y)| \,dy$) to obtain
\begin{align*} 
\norm{f- K_\beta \ast f}_{L^2(\R^d)}^2 & = \int_{\R^d} \left|\int_{\R^d}  \left( f(x- \beta y) - P_{\floor{s}}(x, -\beta y)\right)K(y)\,dy\right|^2 dx \\
& \leq \int_{\R^d} \left(\int_{\R^d}  \left| f(x- \beta y)- P_{\floor{s}}(x, -\beta y) \right| |K(y)|\,dy\right)^2 dx \\
& \leq \left(\int_{\R^d}\left( \int_{\R^d}  \left|f(x- \beta y) - P_{\floor{s}}(x, -\beta y) \right|^p\,dx\right)^{1/2} |K(y)| \, dy\right)^2 \\
& = \left(\int_{\R^d} \norm{ f(\cdot - \beta y)- P_{\floor{s}}(\cdot, -\beta y)}_{L^2(\R^d)} |K(y)| \, dy \right)^2.
\end{align*}
We may now use the estimate for the remainder term in the Taylor expansion~\cite{FractionalSobolev2020}, which for a function in $f \in H^s(\R^d)$ reads
\[\forall z \in \R^d, \qquad \norm{ f(\cdot -z) -P_{\floor{s}}(\cdot, z)}_{L^2(\R^d)} \lesssim  |z|^s \|f\|_{H^s(\R^d)}.\]
We end up with 
\begin{align*} \norm{f- K_\beta \ast f}_{L^2(\R^d)} & \lesssim \beta^{s} \left(\int_{\R^d} |y|^s |K(y)| \, dy \right) \|f\|_{H^s(\R^d)}^2 \lesssim \beta^{s} \|f\|_{H^s(\R^d)}  \lesssim \beta^{s_r} \|f\|_{H^s(\R^d)},
\end{align*}
where we used $s \geq s_r$.
The result is proved for $\Omega = \R^d$. It remains to consider the case where $\Omega$ is a smooth domain.
For $f \in H^s_0(\Omega)$, its extension $\tilde f$ by $0$ to the whole of $\R^d$ satisfies $\tilde f \in H^s(\R^d)$, in which case one may use the above estimate
\[ \norm{K_\beta \ast  \tilde f - \tilde f}_{L^2(\R^d)} \lesssim  \|\tilde f\|_{H^{s_r}(\R^d)}  \beta^{s_r} =\|f\|_{s_r}  \beta^{s_r}.\]
This in turn leads to a bound for the error between $K_\beta \ast f$ and $f$ in~$L^2(\Omega)$,
\begin{equation*}
 \norm{K_\beta \ast  f - f} \leq\norm{K_\beta \ast  \tilde f - \tilde f}_{L^2(\R^d)} \lesssim  \|f\|_{s_r}  \beta^{s_r},
\end{equation*}
and concludes the proof.
\qed

\section{Further estimates}
\label{app_further}
\begin{lemma}
\label{basic_interaction}
Assume that~\eqref{sup_base} holds, and that $K \geq 0$. Then if $\beta = \beta(h) = o(h)$, there holds $\|K_\beta \ast \phi_i\| \gtrsim h^{d/2}$.
\end{lemma}
\begin{proof}
Let $0< r < c$ be fixed with $c$ given by~\eqref{sup_base}. We also pick $M>0$ such that  $K(z) = 0$ for $|z| > M$.
Let us evaluate $K_\beta \ast \phi_i(x)$ for $x \in B(x_i, r h) \cap \Omega$. 
For any $y \in \Omega$,  we shall prove that $K_\beta(x-y) > 0 \implies y \in B(x_i, ch)$. 
Indeed, the first condition imposes $|y-x| \leq M \beta$, hence $|y-x_i| \leq |y-x| + |x-x_i| \leq  M \beta + r h \leq c h$ for $h$ small enough since  $\beta = o(h)$. 
Owing to $K\geq 0$, this allows us to write for $x \in B(x_i, r h)$
\[K_\beta \ast \phi_i(x) = \int_\Omega K_\beta (x-y) \phi_i(y) \, dy \geq m \int_\Omega K_\beta(x-y)\,dy = m \int_{\beta^{-1}(x- \Omega)}K(z)\,dz.\]
For a given $x \in B(x_i, rh)$, $\beta^{-1} (x-\Omega)$ contains a ball of the form $\{z \in \R^d,\, |z| \leq \e \beta^{-1} h\}$ for some $\e$ small enough, and since $1 = o(\beta^{-1} h)$, the latter ball contains the support of $K$ for $h$ small enough, leading to
\[K_\beta \ast \phi_i(x)=  m \int_{\beta^{-1}(x- \Omega)}K(z)\,dz \gtrsim m \int_{\R^d} K(z)\,dz \gtrsim 1.\]
where $\gtrsim$ is uniform with respect to $x \in B(x_i, rh)$.
We conclude that
\[\|K_\beta \ast \phi_i\|^2 \geq \int_{B(x_i, rh)} |K_\beta \ast \phi_i(x)|^2 \,dx \gtrsim |B(x_i, rh)|  \sim h^d.\]
\end{proof}

\begin{lemma}
Under the assumptions of~Lemma~\ref{basic_interaction}, if $\beta(h) = o(h)$, we have
\[e(\beta(h), \sigma,h) \gtrsim  \sigma.\]
\end{lemma}
\begin{proof}
This is a direct consequence of Lemma~\ref{basic_interaction}, since one then has
\[e(\beta, \sigma,h)^2   \geq \sigma^2 \sum_{i=1}^n \norm{ K_\beta \ast \phi_i}_2^2 \gtrsim\sigma^2 \sum_{i=1}^n h^d = \sigma^2 n h^d  \sim \sigma^2.\]

\end{proof}
%
%
%
%
%
%
%

\end{document}